\documentclass[10pt]{amsart}

\usepackage{enumerate}
\usepackage{amsmath,amssymb}

\newtheorem{theorem}{Theorem}[section]
\newtheorem{lemma}[theorem]{Lemma}
\newtheorem{problem}[theorem]{Problem}

\def\bM{{\mathbb{M}}}

\def\cD{{\mathcal{D}}}
\def\cF{{\mathcal{F}}}
\def\cG{{\mathcal{G}}}
\def\cM{{\mathcal{M}}}
\def\cR{{\mathcal{R}}}

\begin{document}

\title[Correlation matrices]{On the shape of correlation matrices for unitaries}

\author[M. Mori]{Michiya Mori}

\address{Graduate School of Mathematical Sciences, The University of Tokyo, 3-8-1 Komaba, Meguro-ku, Tokyo, 153-8914, Japan; Interdisciplinary Theoretical and Mathematical Sciences Program (iTHEMS), RIKEN, 2-1 Hirosawa, Wako, Saitama, 351-0198, Japan.}
\email{mmori@ms.u-tokyo.ac.jp}

\thanks{The author is supported by JSPS KAKENHI Grant Number 22K13934.}


\date{}

\begin{abstract}
For a positive integer $n$, we study the collection $\mathcal{F}_{\mathrm{fin}}(n)$ formed of all $n\times n$ matrices whose entries $a_{ij}$, $1\leq i,j\leq n$, can be written as $a_{ij}=\tau(U_j^*U_i)$ for some $n$-tuple $U_1, U_2, \ldots, U_n$ of unitaries in a finite-dimensional von Neumann algebra $\mathcal{M}$ with tracial state $\tau$. 
We show that $\mathcal{F}_{\mathrm{fin}}(n)$ is not closed for every $n\geq 8$. 
This improves a result by Musat and R{\o}rdam which states the same for $n\geq 11$.
\end{abstract}

\maketitle
\thispagestyle{empty}

\section{Introduction}
The Connes Embedding Problem concerning ``approximation'' of a type II$_1$ von Neumann algebra by finite-dimensional ones has long been regarded as one of the most important problems in the theory of operator algebras. 
By now this problem is known to be equivalent to various other problems. 
One of these problems is claimed to be resolved in the preprint \cite{JNVWY}, and it implies that the Connes Embedding Problem has a negative answer. 
However, the difference between the world of general type II$_1$ von Neumann algebras and the finite-dimensional world is still quite mysterious.

Let $n$ be a positive integer. 
We consider the collection $\cD(n)$ (resp.\ $\cD_{\mathrm{fin}}(n)$) formed of all $n\times n$ matrices whose entries $a_{ij}$, $1\leq i,j\leq n$, can be written as $a_{ij}=\tau(P_iP_j)$ for some $n$-tuple $P_1, P_2, \ldots, P_n$ of projections in a finite von Neumann algebra (resp.\ finite-dimensional von Neumann algebra) $\cM$ with (normal faithful) tracial state $\tau$. 
It is known that the existence of $n$ such that $\cD_{\mathrm{fin}}(n)$ is not dense in $\cD(n)$ is equivalent to the negative answer to the Connes Embedding Problem.
Likewise, we may consider unitaries instead of projections. 
Let $\cG(n)$ (resp.\ $\cF_{\mathrm{fin}}(n)$) denote the set of all $n\times n$ matrices whose entries $a_{ij}$, $1\leq i,j\leq n$, can be written as $a_{ij}=\tau(U_j^*U_i)$ for some $n$-tuple $U_1, U_2, \ldots, U_n$ of unitaries in a finite von Neumann algebra (resp.\ finite-dimensional von Neumann algebra) $\cM$ with tracial state $\tau$. 
Then the existence of $n$ such that $\cF_{\mathrm{fin}}(n)$ is not dense in $\cG(n)$ is equivalent to the negative answer to the Connes Embedding Problem.
More information concerning the sets $\cD(n)$, $\cD_{\mathrm{fin}}(n)$, $\cG(n)$, and $\cF_{\mathrm{fin}}(n)$ (their motivation, the relation to the Connes Embedding Problem and quantum information theory, and applications to the theory of von Neumann algebras) is detailed in \cite{MR}.

It is natural to pursue a better understanding of the shapes of $\cD(n)$, $\cD_{\mathrm{fin}}(n)$, $\cG(n)$, and $\cF_{\mathrm{fin}}(n)$ for a small $n$. 
It is easily seen that these sets are bounded and convex. 
A technique of ultraproduct enables us to show that $\cD(n)$ and $\cG(n)$ are compact.
The following problems are of particular interest.
\begin{problem}
What is the smallest $n$ such that $\cD_{\mathrm{fin}}(n)$ is not compact? 
\end{problem}
\begin{problem}
What is the smallest $n$ such that $\cD_{\mathrm{fin}}(n)$ is not dense in $\cD(n)$? 
\end{problem}
\begin{problem}
What is the smallest $n$ such that $\cF_{\mathrm{fin}}(n)$ is not compact? 
\end{problem}
\begin{problem}
What is the smallest $n$ such that $\cF_{\mathrm{fin}}(n)$ is not dense in $\cG(n)$? 
\end{problem}

Let $n_k$ denote the solution of Problem 1.$k$, $k=1,2,3,4$. 
Recall that the negative answer to the Connes Embedding Problem implies $n_2, n_4<\infty$.
Dykema and Juschenko proved that every point of $\cG(n)$ comes from an $n$-tuple of commuting unitaries if $n\leq 3$ \cite[Corollary 2.8]{DJ}. 
This implies $n_3, n_4\geq 4$.
As for $n_1$ and $n_2$, Dykema--Paulsen--Prakash obtained the upper bound $n_1\leq 5$ in \cite{DPP} (see also \cite{MR}). 
Musat and R{\o}rdam applied it to show $n_3\leq 11$ \cite[Theorem 3.6]{MR}.
They also gave $n_1, n_2\geq 3$ \cite[Proposition 2.1]{MR}.
In \cite[Theorem 4.8]{R}, Russell proved that $n_1, n_2\geq 4$.

Our main result is that $n_3\leq 8$.
Thus, after our contribution we know 
\[
4\leq n_1\leq 5,\quad 4\leq n_2,\quad 4\leq n_3\leq 8,\quad\text{and}\quad 4\leq n_4.
\]
Note that the proof of $n_1\leq 5$ in \cite{DPP,MR} depends on the structure of finite sums projections given by Kruglyak, Rabanovich, and Samo\u{\i}lenko \cite{KRS}.
Our proof of $n_3\leq 8$ uses the structure of finite products of self-adjoint unitaries.

\section{Results}
Our goal in the rest of this article is to show that $\cF_{\mathrm{fin}}(n)$ is not closed if $n\geq 8$. 
\begin{lemma}\label{sach}
Let $\cM$ be a von Neumann algebra with faithful tracial state $\tau$.
Let $U \in \cM$ be a unitary. 
Then the following are equivalent. 
\begin{enumerate}
\item $U=U^*$.
\item The operator $(U+iI)/\sqrt{2}$ is also a unitary.
\item There is a unitary operator $V\in M$ such that 
\[
\sqrt{2}\operatorname{Re} \tau(U^*V)-\operatorname{Im}\tau(U-\sqrt{2}V)=2.
\] 
\end{enumerate}
\end{lemma}
\begin{proof}
$(1)\Leftrightarrow(2)$ This is clear by looking at the spectrum.\\
$(2)\Leftrightarrow(3)$ If $V\in \cM$ is unitary, then 
\[
(U+iI-\sqrt{2}V)^*(U+iI-\sqrt{2}V)=4I-iU+iU^*+\sqrt{2}iV-\sqrt{2}iV^*-\sqrt{2}U^*V-\sqrt{2}V^*U.
\] 
It follows that 
\[
\frac{1}{2}\tau((U+iI-\sqrt{2}V)^*(U+iI-\sqrt{2}V)) =2+\operatorname{Re} \tau(-iU+\sqrt{2}iV-\sqrt{2}U^*V).
\]
Since $\tau$ is faithful, the condition $U+iI-\sqrt{2}V=0$ is equivalent to $2+\operatorname{Re} \tau(-iU+\sqrt{2}iV-\sqrt{2}U^*V)=0$, and this leads to the desired conclusion.
\end{proof}

We use the following well-known fact (see \cite[Proposition 3.12]{BNS}).
\begin{lemma}
If $\cM$ is a von Neumann algebra of type II$_1$ and $\kappa\in \mathbb{R}$, then there are self-adjoint unitaries $S_1, S_2, S_3, S_4\in \cM$ such that $S_1S_2S_3S_4=e^{2\pi i \kappa}I$.
\end{lemma}
On the other hand, in the matrix algebra $\bM_n$, any finite product of self-adjoint unitaries has determinant $\pm 1$. 
Therefore, in a finite-dimensional von Neumann algebra, a finite product of self-adjoint unitaries is never equal to the  operator $e^{2\pi i \kappa}I$ for an irrational $\kappa$.

\begin{theorem}\label{main}
If $n\geq 8$, then the set $\cF_{\mathrm{fin}}(n)$ is not closed.
\end{theorem}
\begin{proof}
Fix an irrational number $\kappa$ and self-adjoint unitaries $S_1, S_2, S_3, S_4$ in the AFD II$_1$ factor $(\cR, \tau_\cR)$ with $S_1S_2S_3S_4=e^{2\pi i \kappa}I$. 
We set 
\[
U_1:=I,\quad U_2:=S_1,\quad U_3:=S_1S_2,\quad U_4:=S_1S_2S_3=e^{2\pi i \kappa}S_4,
\]
\[
U_5:=\frac{S_1+iI}{\sqrt{2}},\quad U_6:=\frac{S_1(S_2+iI)}{\sqrt{2}},\quad U_7:=\frac{S_1S_2(S_3+iI)}{\sqrt{2}},\quad U_8:=\frac{S_4+iI}{\sqrt{2}},
\] 
and $U_k:=I$ for all $9\leq k\leq n$. 
Since $\cR$ is AFD, the matrix $(\tau(U_j^*U_i))_{1\leq i,j\leq n}$ belongs to the closure of $\cF_{\mathrm{fin}}(n)$. 
We show that this matrix is not in $\cF_{\mathrm{fin}}(n)$.

Assume (towards a contradiction) that there are unitaries $V_j$, $j=1,2,\ldots, 8$, in a finite-dimensional von Neumann algebra $\cM$ with faithful tracial state $\tau$ such that $\tau(V_j^*V_i)=\tau_\cR(U_j^*U_i)$, $1\leq i,j\leq 8$. 
For $j=1,2,3$, we have
\[
\begin{split}
&\quad\,\, \sqrt{2}\operatorname{Re} \tau((V_{j}^*V_{j+1})^*(V_j^*V_{j+4}))-\operatorname{Im}\tau((V_{j}^*V_{j+1})-\sqrt{2}(V_j^*V_{j+4}))\\
&= \sqrt{2}\operatorname{Re} \tau(V_{j+1}^*V_{j+4})-\operatorname{Im}(\tau(V_{j}^*V_{j+1})-\sqrt{2}\tau(V_j^*V_{j+4}))\\
&= \sqrt{2}\operatorname{Re} \tau_\cR(U_{j+1}^*U_{j+4})-\operatorname{Im}(\tau_\cR(U_{j}^*U_{j+1})-\sqrt{2}\tau_\cR(U_j^*U_{j+4}))\\
&= \sqrt{2}\operatorname{Re} \tau_\cR\left(\frac{I+iS_j}{\sqrt{2}}\right)-\operatorname{Im}\left(\tau_\cR(S_j)-\sqrt{2}\tau_\cR\left(\frac{S_j+iI}{\sqrt{2}}\right)\right)\\
&= 2.
\end{split}
\] 
By Lemma \ref{sach}, $V_{1}^*V_{2}$, $V_2^*V_3$, $V_3^*V_4$ are self-adjoint.
Similarly, the equality
\[
\begin{split}
&\quad\,\, \sqrt{2}\operatorname{Re} \tau((e^{-2\pi i\kappa}V_1^*V_4)^*(V_1^*V_8))-\operatorname{Im}\tau((e^{-2\pi i\kappa}V_1^*V_4)-\sqrt{2}(V_1^*V_8))\\
&= \sqrt{2}\operatorname{Re} e^{2\pi i\kappa}\tau(V_4^*V_8)-\operatorname{Im}(e^{-2\pi i\kappa}\tau(V_1^*V_4)-\sqrt{2}\tau(V_1^*V_8))\\
&= \sqrt{2}\operatorname{Re} e^{2\pi i\kappa}\tau_\cR(U_4^*U_8)-\operatorname{Im}(e^{-2\pi i\kappa}\tau_\cR(U_1^*U_4)-\sqrt{2}\tau_\cR(U_1^*U_8))\\
&= \sqrt{2}\operatorname{Re} e^{2\pi i\kappa}\tau_\cR\left(\frac{e^{-2\pi i\kappa}(I+iS_4)}{\sqrt{2}}\right)-\operatorname{Im}\left(e^{-2\pi i\kappa}\tau_\cR(e^{2\pi i\kappa}S_4)-\sqrt{2}\tau_\cR\left(\frac{S_4+iI}{\sqrt{2}}\right)\right)\\
&= 2
\end{split}
\]
implies that $e^{-2\pi i\kappa}V_{1}^*V_{4}$ is also self-adjoint. 
Therefore, the operator $e^{2\pi i\kappa}I$ decomposes into the product $(V_{1}^*V_{2})\cdot(V_2^*V_3)\cdot (V_3^*V_4)\cdot (e^{-2\pi i\kappa}V_{1}^*V_{4})^*$ of four self-adjoint unitaries.
This is impossible because $\cM$ is finite-dimensional.
\end{proof}

By imitating the proof of \cite[Theorem 4.1]{MR}, we see that Theorem \ref{main} implies certain unfactorizability of some unital completely positive trace preserving mapping on $\bM_n$ for every $n\geq 8$. See \cite{MR} for the details.

It seems plausible for the author to guess that $n_1\in \{4,5\}$ is close to $n_3$ (and $n_2$ is close to $n_4$). 
Unfortunately, the above discussion, which heavily relies on the use of self-adjoint unitaries, apparently does not give an upper bound of $n_3$ that is better than $8$.

\medskip\medskip

\noindent 
\textbf{Acknowledgements.}\,
The author appreciates Travis B. Russell (Texas Christian University) for notifying the author that part of the results in the previous version of this article was already given in \cite{R} with similar techniques.


\begin{thebibliography}{00}
\bibitem{BNS} B.V.R. Bhat, S. Nayak, and P. Shankar, On products of symmetries in von Neumann algebras. To appear in \emph{J. Operator Theory}.
\bibitem{DJ} K. Dykema and K. Juschenko, Matrices of unitary moments. \emph{Math. Scand.} \textbf{109} (2011), no.2, 225--239.
\bibitem{DPP}  K. Dykema, V. Paulsen, and J. Prakash, Non-closure of the set of quantum correlations via graphs. \emph{Comm. Math. Phys.} \textbf{365} (2019), no. 3, 1125--1142.
\bibitem{JNVWY} Z. Ji, A. Natarajan, T. Vidick, J. Wright, and H. Yuen, $\mathrm{MIP}^*=\mathrm{RE}$. Preprint, arXiv:2001.04383.
\bibitem{KRS} S.A. Kruglyak, V.I. Rabanovich, and Y.S. Samo\u{\i}lenko, On sums of projections. \emph{Funktsional. Anal. i Prilozhen.} \textbf{36} (2002), no. 3, 20--35, 96; translation in \emph{Funct. Anal. Appl.} \textbf{36} (2002), no. 3, 182--195.
\bibitem{MR} M. Musat and M. R{\o}rdam, Non-closure of quantum correlation matrices and factorizable channels that require infinite dimensional ancilla. With an appendix by Narutaka Ozawa, \emph{Comm. Math. Phys.} \textbf{375} (2020), no. 3, 1761--1776.
\bibitem{R} T.B. Russell, Two--outcome synchronous correlation sets and Connes' embedding problem. \emph{Quantum Inf. Comput.} \textbf{20} (2020), no. 5--6, 361--374.
\end{thebibliography}
\end{document}